\documentclass[10pt]{amsart}
\usepackage[psamsfonts]{amssymb}
\usepackage{amsthm,amsmath}

\title[Measuring the influence of the $k$th largest variable]{Measuring the influence of the $k$th largest variable on functions over the unit hypercube}

\author{Jean-Luc Marichal}
\address{Mathematics Research Unit, FSTC, University of Luxembourg, 6, rue Coudenhove-Kalergi, L-1359 Luxembourg, Grand Duchy of Luxembourg.}
\email{jean-luc.marichal[at]uni.lu}

\author{Pierre Mathonet}
\address{Mathematics Research Unit, FSTC, University of Luxembourg, 6, rue Coudenhove-Kalergi, L-1359 Luxembourg, Grand Duchy of Luxembourg.}
\email{pierre.mathonet[at]uni.lu}

\date{March 12, 2010}

\begin{document}

\theoremstyle{plain}
\newtheorem{theorem}{Theorem}
\newtheorem{lemma}[theorem]{Lemma}
\newtheorem{proposition}[theorem]{Proposition}
\newtheorem{corollary}[theorem]{Corollary}
\newtheorem{fact}[theorem]{Fact}
\newtheorem*{main}{Main Theorem}

\theoremstyle{definition}
\newtheorem{definition}[theorem]{Definition}
\newtheorem{example}[theorem]{Example}

\theoremstyle{remark}
\newtheorem*{conjecture}{\indent Conjecture}
\newtheorem{remark}{Remark}
\newtheorem{claim}{Claim}

\newcommand{\N}{\mathbb{N}}
\newcommand{\R}{\mathbb{R}}
\newcommand{\I}{\mathbb{I}}
\newcommand{\Vspace}{\vspace{2ex}}
\newcommand{\bfx}{\mathbf{x}}
\newcommand{\bfy}{\mathbf{y}}
\newcommand{\bfh}{\mathbf{h}}
\newcommand{\bfe}{\mathbf{e}}
\newcommand{\Q}{Q}
\newcommand{\os}{\mathrm{os}}

\begin{abstract}
By considering a least squares approximation of a given square integrable function $f\colon [0,1]^n\to\R$ by a shifted $L$-statistic function (a
shifted linear combination of order statistics), we define an index which measures the global influence of the $k$th largest variable on $f$. We
show that this influence index has appealing properties and we interpret it as an average value of the difference quotient of $f$ in the
direction of the $k$th largest variable or, under certain natural conditions on $f$, as an average value of the derivative of $f$ in the
direction of the $k$th largest variable. We also discuss a few applications of this index in statistics and aggregation theory.
\end{abstract}



\maketitle

\section{Introduction}

Consider a real-valued function $f$ of $n$ variables $x_1,\ldots,x_n$ and suppose we want to measure a global influence degree of every variable
$x_i$ on $f$. A reasonable way to define such an influence degree consists in considering the coefficient of $x_i$ in the best least squares
approximation of $f$ by affine functions of the form
$$
g(x_1,\ldots,x_n)=c_0+\sum_{i=1}^n c_i x_i.
$$
This approach was considered in \cite{HamHol92,MarMatb} for pseudo-Boolean functions\footnote{An alternative (but equivalent) definition of influence index was previously considered for Boolean functions in \cite{KahKalLin88} and pseudo-Boolean functions in \cite{Mar00}.} $f\colon\{0,1\}^n\to\R$ and in \cite{MarMat} for square integrable functions
$f\colon [0,1]^n\to\R$. It turns out that, in both cases, the influence index of $x_i$ on $f$ is given by an average ``derivative'' of $f$ with
respect to $x_i$.

Now, it is also natural to consider and measure a global influence degree of the smallest variable, or the largest variable, or even the $k$th
largest variable for some $k\in\{1,\ldots,n\}$. As an application, suppose we are to choose an appropriate aggregation function $f\colon
[0,1]^n\to\R$ to compute an average value of $[0,1]$-valued grades obtained by a student. If, for instance, we use the arithmetic mean function,
we might expect that both the smallest and the largest variables are equally influent. However, if we use the geometric mean function, for which
the value $0$ (the left endpoint of the scale) is multiplicatively absorbent, we might anticipate that the smallest variable is more influent
than the largest one.

Similarly to the previous problem, to define the influence of the $k$th largest variable on $f$ it is natural to consider the coefficient of
$x_{(k)}$ in the best least squares approximation of $f$ by symmetric functions of the form
$$
g(x_1,\ldots,x_n)=a_0+\sum_{i=1}^n a_i x_{(i)},
$$
where $x_{(1)},\ldots,x_{(n)}$ are the order statistics obtained by rearranging the variables in ascending order of magnitude.

In this paper we solve this problem for square integrable functions $f\colon [0,1]^n\to\R$. More precisely, we completely describe the least
squares approximation problem above and derive an explicit expression for the corresponding influence index ({\S}2). We also show that this
index has several natural properties, such as linearity and continuity, and we give an interpretation of it as an average value of the
difference quotient of $f$ in the direction of the $k$th largest variable. Under certain natural conditions on $f$, we also interpret the index
as an average value of the derivative of $f$ in the direction of the $k$th largest variable ({\S}3). We then provide some alternative formulas
for the index to possibly simplify its computation ({\S}4) and we consider some examples including the case when $f$ is the Lov\'asz extension
of a pseudo-Boolean function ({\S}5). Finally, we discuss a few applications of the index ({\S}6).

We employ the following notation throughout the paper. Let $\I^n$ denote the $n$-dimensional unit cube $[0,1]^n$. We denote by $L^2(\I^n)$ the
class of square integrable functions $f\colon \I^n\to\R$ modulo equality almost everywhere. For any $S\subseteq [n]=\{1,\ldots,n\}$, we denote
by $\mathbf{1}_S$ the characteristic vector of $S$ in $\{0,1\}^n$ (with the particular case $\mathbf{0}=\mathbf{1}_{\varnothing}$).

Recall that if the $\I$-valued variables $x_1,\ldots,x_n$ are rearranged in ascending order of magnitude $ x_{(1)}\leqslant\cdots\leqslant
x_{(n)}, $ then $x_{(k)}$ is called the \emph{$k$th order statistic} and the function $\os_k\colon\I^n\to\R$, defined as $\os_k(\bfx)=x_{(k)}$,
is the \emph{$k$th order statistic function}. As a matter of convenience, we also formally define $\os_0\equiv 0$ and $\os_{n+1}\equiv 1$. To
stress on the arity of the function, we can replace the symbols $x_{(k)}$  and $\os_k$ with $x_{k:n}$ and $\os_{k:n}$, respectively. For general
background on order statistics, see for instance \cite{BalRao98,DavNag03}.

Finally, we use the lattice notation $\wedge$ and $\vee$ to denote the minimum and maximum functions, respectively.

\section{Influence index for the $k$th largest variable}

An \emph{$L$-statistic} function is a linear combination of the functions $\os_1,\ldots,\os_n$. A \emph{shifted $L$-statistic} function is a
constant plus an $L$-statistic function. Denote by $V_L$ the set of shifted $L$-statistic functions. Clearly, $V_L$ is spanned by the linearly
independent set
\begin{equation}\label{eq:Base21331}
B=\{\os_1,\ldots,\os_n,\os_{n+1}\}
\end{equation}
and thus is a linear subspace of $L^2(\I^n)$ of dimension $n+1$. For a given function $f\in L^2(\I^n)$, we define the \emph{best shifted
$L$-statistic approximation of $f$} as the function $f_L\in V_L$ that minimizes the distance
$$
\lVert f-g\rVert^2=\int_{\I^n}\big(f(\bfx)-g(\bfx)\big)^2\, d\bfx
$$
among all $g\in V_L$, where $\|\cdot\|$ is the norm in $L^2(\I^n)$ associated with the inner product $\langle
f,g\rangle=\int_{\I^n}f(\bfx)g(\bfx)\, d\bfx$. Using the general theory of Hilbert spaces, we immediately see that the solution of this
approximation problem exists and is uniquely determined by the orthogonal projection of $f$ onto $V_L$. This projection is given by
\begin{equation}\label{eq:a2s3d1as}
f_L=\sum_{j=1}^{n+1}a_j\,\os_j\, ,
\end{equation}
where the coefficients $a_j$ (for $j\in [n+1]$) are characterized by the conditions
\begin{equation}\label{eq:1s23f1s}
\langle f-f_L,\os_i\rangle =0\quad\mbox{for all}\quad i\in [n+1].
\end{equation}
Consider the matrix representing the inner product in the basis (\ref{eq:Base21331}), that is, the square matrix $M$ of order $n+1$ defined by
$(M)_{ij}=\langle \os_i,\os_j\rangle$ for all $i,j\in [n+1]$. Denote also by $\mathbf{b}$ the $(n+1)\times 1$ column matrix defined by
$(\mathbf{b})_i=\langle f,\os_i\rangle$ for all $i\in [n+1]$ and by $\mathbf{a}$ the $(n+1)\times 1$ column matrix defined by
$(\mathbf{a})_j=a_j$ for all $j\in [n+1]$. Using this notation, the unique solution of the approximation problem defined in (\ref{eq:a2s3d1as})
and (\ref{eq:1s23f1s}) is simply given by
\begin{equation}\label{eq:45xycv}
\mathbf{a}=M^{-1}\mathbf{b}.
\end{equation}

To give an explicit expression of this solution, we shall make use of the following formula (see \cite{DavJoh54}). For any integers $1\leqslant
k_1<\cdots <k_m\leqslant n$ and any nonnegative integers $c_1,\ldots,c_m$, we have
\begin{equation}\label{eq:sd87f9}
\int_{\I^n}\prod_{j=1}^mx_{k_j:n}^{c_j}\,
d\bfx=\frac{n!}{\big(n+\sum_{j=1}^mc_j\big)!}\,\prod_{j=1}^m\frac{\big(k_j-1+\sum_{i=1}^jc_i\big)!}{\big(k_j-1+\sum_{i=1}^{j-1}c_i\big)!}\, .
\end{equation}

\begin{lemma}\label{lemma:safadfd456}
For every $i,j\in [n+1]$, we have
\begin{equation}\label{eq:w9rer89}
(M)_{ij} = \frac{\min(i,j)\big(\max(i,j)+1\big)}{(n+1)(n+2)}
\end{equation}
and
\begin{equation}\label{eq:w9r89}
\frac{(M^{-1})_{ij}}{(n+1)(n+2)}=
\begin{cases}
2\, , & \mbox{if $i=j < n+1$},\\
\frac{n+1}{n+2}\, , & \mbox{if $i=j=n+1$},\\
-1\, , & \mbox{if $|i-j|=1$},\\
0\, , & \mbox{otherwise}.
\end{cases}
\end{equation}
\end{lemma}

\begin{proof}
The formula for $(M)_{ij} = \langle \os_i,\os_j\rangle$ immediately follows from (\ref{eq:sd87f9}). The formula for $(M^{-1})_{ij}$ can be
checked easily.
\end{proof}

Recall that the central second difference operator is defined for any real sequence $(z_k)_{k\geqslant 1}$ as $\delta_k^2\,
z_k=z_{k+1}-2z_k+z_{k-1}$. For every $k\in [n]$, define the function $g_k\in L^2(\I^n)$ as
\begin{equation}\label{eq:s7d9f}
g_k=-(n+1)(n+2)\,\delta_k^2\,\os_k.
\end{equation}
Using (\ref{eq:45xycv}) and (\ref{eq:w9r89}), we immediately obtain the following explicit forms for the components of $f_L$ in the basis
(\ref{eq:Base21331}).

\begin{proposition}
The best shifted $L$-statistic approximation $f_L$ of a function $f\in L^2(\I^n)$ is given by (\ref{eq:a2s3d1as}), where
\begin{equation}\label{eq:bx3cv43}
a_k=
\begin{cases}
\langle f,g_k\rangle\, , & \mbox{if $k\in [n]$},\\
(n+1)^2\langle f,1\rangle-(n+1)(n+2)\langle f,\os_n\rangle\, , & \mbox{if $k=n+1$}.
\end{cases}
\end{equation}
\end{proposition}

Now, to measure the global influence of the $k$th largest variable $x_{(k)}$ on an arbitrary function $f\in L^2(\I^n)$, we naturally define an index
$I\colon L^2(\I^n)\times [n]\to\R$ as $I(f,k)=a_k$, where $a_k$ is obtained from $f$ by (\ref{eq:bx3cv43}). We will see in the next section that
this index indeed measures an influence degree.

\begin{definition}\label{de:dfdgf78}
Let $I\colon L^2(\I^n)\times [n]\to\R$ be defined as $I(f,k)=\langle f,g_k\rangle$, that is
\begin{equation}\label{eq:dfdgf78}
I(f,k)=-(n+1)(n+2)\int_{\I^n}f(\bfx)\, \delta_k^2\, x_{(k)}\, d\bfx.
\end{equation}
\end{definition}

\begin{remark}
By combining (\ref{eq:sd87f9}) and (\ref{eq:dfdgf78}), we see that the index $I(f,k)$ can be easily computed when $f$ is any polynomial function
of order statistics.
\end{remark}

Thus we have defined an influence index from an elementary approximation (projection) problem. Conversely, the following result shows that the
best shifted $L$-statistic approximation of $f\in L^2(\I^n)$ is the unique function of $V_L$ that preserves the average value and the influence
index. To this extent, we observe that letting $i=n+1$ in (\ref{eq:1s23f1s}) leads to $\langle f_L,1\rangle=\langle f,1\rangle$, that is,
\begin{equation}\label{eq:6sa6df}
\frac{1}{n+1}\sum_{k=1}^{n+1}k\, a_k=\langle f,1\rangle.
\end{equation}

\begin{proposition}
A function $g\in V_L$ is the best shifted $L$-statistic approximation of $f\in L^2(\I^n)$ if and only if
$\int_{\I^n}f(\bfx)\,d\bfx=\int_{\I^n}g(\bfx)\,d\bfx$ and $I(f,k)=I(g,k)$ for all $k\in [n]$.
\end{proposition}

\begin{proof}
We formally extend $I(f,\cdot)$ to $[n+1]$ by defining $I(f,n+1)=a_{n+1}$. By (\ref{eq:1s23f1s}), the function $g\in V_L$ is the best shifted
$L$-statistic approximation of $f\in L^2(\I^n)$ if and only if $\langle f,\os_i\rangle = \langle g,\os_i\rangle$ for all $i\in [n+1]$. By
(\ref{eq:45xycv}), this condition is equivalent to $I(f,k)=I(g,k)$ for all $k\in [n+1]$. We then conclude by (\ref{eq:6sa6df}).
\end{proof}

\begin{remark}
Combining (\ref{eq:a2s3d1as}) with (\ref{eq:6sa6df}), we can rewrite the best shifted $L$-statistic approximation of $f$ as $f_L=\langle
f,1\rangle + \sum_{k=1}^n I(f,k)\big(\textstyle{x_{(k)}-\frac{k}{n+1}}\big)$.
\end{remark}

\section{Properties and interpretations}

In this section we present various properties and interpretations of the index $I(f,k)$. The first result follows immediately from
Definition~\ref{de:dfdgf78}.

\begin{proposition}\label{prop:asd7fd}
For every $k\in [n]$, the mapping $f\mapsto I(f,k)$ is linear and continuous.
\end{proposition}

We now present an interpretation of $I(f,k)$ as a covariance. Considering the unit cube $\I^n$ as a probability space with respect to the
Lebesgue measure, we see that, for any $k\in [n]$, the index $I(f,k)$ is the covariance of the random variables $f$ and $g_k$. Indeed, we have
$I(f,k)=E(f\, g_k)=\mathrm{cov}(f,g_k)+E(f)\, E(g_k)$, where $E(g_k)=\langle 1,g_k\rangle = I(1,k)=0$. From the usual interpretation of the
concept of covariance, we see that $I(f,k)$ is positive whenever the values of $f-E(f)$ and $g_k-E(g_k)=g_k$ have the same sign. Note that $g_k(\bfx)$ is
positive whenever $x_{(k)}$ is greater than $\frac 12(x_{(k+1)}+x_{(k-1)})$, which is the midpoint of the range of $x_{(k)}$ when the other
order statistics are fixed at $\bfx$.

We now provide an interpretation of $I(f,k)$ as an expected value of the derivative of $f$ in the direction of the $k$th largest variable (see
Proposition~\ref{prop:sdf79}).

Let $S_n$ denote the symmetric group on $[n]$. Recall that the
unit cube $\I^n$ can be partitioned almost everywhere into the open standard simplexes
$$
\I^n_{\pi}=\{\bfx\in\I^n : x_{\pi(1)}<\cdots < x_{\pi(n)}\} \qquad (\pi\in S_n).
$$

\begin{definition}
Given $k\in [n]$, let $f\colon\cup_{\pi\in S_n}\I^n_{\pi}\to\R$ be a function such that the partial derivative $D_{\pi(k)}f|_{\I^n_{\pi}}$
exists for every $\pi\in S_n$. The \emph{derivative of $f$ in the direction $(k)$} is the function $D_{(k)}f\colon\cup_{\pi\in
S_n}\I^n_{\pi}\to\R$ defined as
$$
D_{(k)}f(\bfx)=D_{\pi(k)}f(\bfx)\quad\mbox{for all}\quad\bfx\in\I^n_{\pi}.
$$
\end{definition}

\begin{remark}
By considering the chain rule in $\cup_{\pi\in S_n}\I^n_{\pi}$ with the usual assumptions, we immediately obtain the formula
$$
D_{(k)}f\big(g_1(\bfx),\ldots,g_n(\bfx)\big)=\sum_{i=1}^n (D_if)\big(g_1(\bfx),\ldots,g_n(\bfx)\big)\, D_{(k)} g_i(\bfx).
$$
\end{remark}



Now, for every $k\in [n]$, consider the function $h_k\in L^2(\I^n)$ defined as
$$
h_k=(n+1)(n+2)(\os_{k+1}-\os_k)(\os_k-\os_{k-1}).
$$
It is immediate to see that $h_k$ is nonnegative and continuous and that $D_{(k)}h_k=-g_k$, where $g_k$ is defined in (\ref{eq:s7d9f}).
Moreover, using (\ref{eq:sd87f9}) or (\ref{eq:w9rer89}), we easily see that $h_k$ is a probability density function on $\I^n$. This fact can
also be derived by choosing $f=\os_k$ in the following result.

\begin{proposition}\label{prop:sdf79}
For every $k\in [n]$ and every $f\in L^2(\I^n)$ such that $D_{(k)}f$ is continuous and integrable on $\cup_{\pi\in S_n}\I^n_{\pi}$, we have
\begin{equation}\label{eq:sdf79}
I(f,k)=\int_{\I^n}h_k(\bfx)\, D_{(k)}f(\bfx)\, d\bfx.
\end{equation}
\end{proposition}

\begin{proof}
Fix $k\in [n]$. Using the product rule, we obtain
$$
h_k(\bfx)\, D_{(k)}f(\bfx)=D_{(k)}\big(h_k(\bfx)\, f(\bfx)\big)+g_k(\bfx)\, f(\bfx),
$$
and hence we only need to show that
\begin{equation}\label{eq:9fs80d}
\int_{\I^n}D_{(k)}\big(h_k(\bfx)\, f(\bfx)\big)\, d\bfx=0.
\end{equation}
But the left-hand side of (\ref{eq:9fs80d}) can be rewritten as
\begin{multline*}
\sum_{\pi\in S_n}\int_{\I^n_{\pi}}D_{\pi(k)}\big(h_k(\bfx)\, f(\bfx)\big)\, d\bfx\\
=\sum_{\pi\in S_n}\int_0^1\int_0^{x_{\pi(n)}}\cdots\int_0^{x_{\pi(3)}}\int_0^{x_{\pi(2)}}D_{\pi(k)}\big(h_k(\bfx)\, f(\bfx)\big)\,
dx_{\pi(1)}dx_{\pi(2)}\cdots dx_{\pi(n)}\, ,
\end{multline*}
that is, if we permute the integrals so that we integrate first with respect to $x_{\pi(k)}$,
$$
\sum_{\pi\in S_n}\int_0^1\cdots\int_{x_{\pi(k-1)}}^{x_{\pi(k+1)}}D_{\pi(k)}\big(h_k(\bfx)\, f(\bfx)\big)\, dx_{\pi(k)}\cdots dx_{\pi(1)}
$$
which is zero since so is the inner integral.
\end{proof}

\begin{remark}
Under the assumptions of Proposition~\ref{prop:sdf79}, if $D_{(k)}f=0$ (resp.\ $\geqslant 0$, $\leqslant 0$) almost everywhere, then $I(f,k)=0$ (resp.\ $\geqslant 0$, $\leqslant 0$).
\end{remark}

We now give an alternative interpretation of $I(f,k)$ as an expected value, which does not require the additional assumptions of
Proposition~\ref{prop:sdf79}. In this more general framework, we naturally replace the derivative with a difference quotient. To this extent, we
introduce some further notation. As usual, we denote by $\mathbf{e}_i$ the $i$th vector of the standard basis for $\R^n$. For every $k\in [n]$
and every $h\in [0,1]$, we define the \emph{$(k)$-difference} (or \emph{discrete $(k)$-derivative}) operator $\Delta_{(k),h}$ over the set of
real functions on $\I^n$ by
$$
\Delta_{(k),h}f(\bfx)=f(\bfx+h\,\mathbf{e}_{\pi(k)})-f(\bfx)
$$
for every $\bfx\in\I^n_{\pi}$ such that $\bfx+h\mathbf{e}_{\pi(k)}\in\I^n_{\pi}$. Thus defined, the value $\Delta_{(k),h}f(\bfx)$ can be
interpreted as the \emph{marginal contribution} of $x_{(k)}$ on $f$ at $\bfx$ with respect to the increase $h$. For instance, we have
$\Delta_{(k),h}\, x_{(k)}=h$.

Similarly, we define the \emph{$(k)$-difference quotient} operator $Q_{(k),h}$ over the set of real functions on $\I^n$ by
$Q_{(k),h}f(\bfx)=\frac{1}{h}\Delta_{(k),h}f(\bfx)$.

\begin{theorem}\label{thm:s7a897}
For every $k\in [n]$ and every $f\in L^2(\I^n)$, we have
\begin{equation}\label{eq:asd7fd}
I(f,k)=(n+1)(n+2)\int_{\I^n}\int_{x_{(k)}}^{x_{(k+1)}}\Delta_{(k),y-x_{(k)}}f(\bfx)\, dy\, d\bfx.
\end{equation}
\end{theorem}

\begin{proof}
The right-hand side of (\ref{eq:asd7fd}) can be rewritten as
\begin{equation}\label{eq:asd7fd1}
(n+1)(n+2)\sum_{\pi\in S_n}\int_{\I^n_{\pi}}\int_{x_{\pi(k)}}^{x_{\pi(k+1)}}\big(f(\bfx+(y-x_{\pi(k)})\,\mathbf{e}_{\pi(k)})-f(\bfx)\big)\, dy\,
d\bfx.
\end{equation}
On the one hand, we have
\begin{equation}\label{eq:asd7fd2}
\int_{\I^n_{\pi}}\int_{x_{\pi(k)}}^{x_{\pi(k+1)}}f(\bfx)\, dy\, d\bfx=\int_{\I^n_{\pi}}(x_{\pi(k+1)}-x_{\pi(k)})f(\bfx)\, d\bfx.
\end{equation}
On the other hand, by permuting the integrals exactly as in the proof of Proposition~\ref{prop:sdf79}, we obtain
\begin{multline*}
\int_{\I^n_{\pi}}\int_{x_{\pi(k)}}^{x_{\pi(k+1)}}f(\bfx+(y-x_{\pi(k)})\,\mathbf{e}_{\pi(k)})\, dy\, d\bfx\\
= \int_0^1\cdots\int_{x_{\pi(k-1)}}^{x_{\pi(k+1)}}\int_{x_{\pi(k)}}^{x_{\pi(k+1)}}f(\bfx+(y-x_{\pi(k)})\,\mathbf{e}_{\pi(k)})\, dy\,
dx_{\pi(k)}\cdots dx_{\pi(1)}
\end{multline*}
which, by permuting the two inner integrals, becomes
\begin{multline*}
\int_0^1\cdots\int_{x_{\pi(k-1)}}^{x_{\pi(k+1)}}\int_{x_{\pi(k-1)}}^{y}f(\bfx+(y-x_{\pi(k)})\,\mathbf{e}_{\pi(k)})\,
dx_{\pi(k)}\, dy\cdots dx_{\pi(1)}\\
= \int_0^1\cdots\int_{x_{\pi(k-1)}}^{x_{\pi(k+1)}}(y-x_{\pi(k-1)})f(\bfx+(y-x_{\pi(k)})\,\mathbf{e}_{\pi(k)})\, dy\cdots dx_{\pi(1)}.
\end{multline*}
By renaming $y$ as $x_{\pi(k)}$, we finally obtain
\begin{equation}\label{eq:asd7fd3}
\int_{\I^n_{\pi}}\int_{x_{\pi(k)}}^{x_{\pi(k+1)}}f(\bfx+(y-x_{\pi(k)})\,\mathbf{e}_{\pi(k)})\, dy\, d\bfx=
\int_{\I^n_{\pi}}(x_{\pi(k)}-x_{\pi(k-1)})f(\bfx)\, d\bfx.
\end{equation}
By substituting (\ref{eq:asd7fd2}) and (\ref{eq:asd7fd3}) in (\ref{eq:asd7fd1}), we finally obtain $I(f,k)$.
\end{proof}

As an immediate consequence of Theorem~\ref{thm:s7a897}, we have the following interpretation of the index $I(f,k)$ as an expected value of a
difference quotient with respect to some distribution.

\begin{corollary}
For every $k\in [n]$ and every $f\in L^2(\I^n)$, we have
$$
I(f,k)=\int_{\I^n}\int_{x_{(k)}}^{x_{(k+1)}}p_k(\bfx,y)\, Q_{(k),y-x_{(k)}}f(\bfx)\, dy\, d\bfx,
$$
where $p_k(\bfx,y)=(n+1)(n+2)(y-x_{(k)})$ defines a probability density function on the set $\{(\bfx,y) : \bfx\in\I^n,\, y\in
[x_{(k)},x_{(k+1)}]\}$.
\end{corollary}

Another important feature of the index is its invariance under the action of permutations. Recall that a permutation $\pi\in S_n$ acts on a
function $f\colon\I^n\to\R$ by $\pi(f)(x_1,\ldots,x_n)=f(x_{\pi(1)},\ldots,x_{\pi(n)})$. By the change of variables theorem, we immediately see
that every $\pi\in S_n$ is an isometry of $L^2(\I^n)$, that is, $\langle \pi(f),\pi(g)\rangle=\langle f,g\rangle$. From this fact, we derive the
following result.

\begin{proposition}\label{prop:sd9f7}
For every $f\in L^2(\I^n)$ and every $\pi\in S_n$, both functions $f$ and $\pi(f)$ have the same best shifted $L$-statistic approximation $f_L$.
Moreover, we have $\|\pi(f)-f_L\|=\|f-f_L\|$.
\end{proposition}

\begin{proof}
Let $f\in L^2(\I^n)$, $g\in V_L$, and $\pi\in S_n$. Since $\pi$ is an isometry of $L^2(\I^n)$ and $g$ is symmetric, by (\ref{eq:1s23f1s}) we
have $\langle \pi(f),g\rangle=\langle f,g\rangle=\langle f_L,g\rangle$, which shows that $\pi(f)_L=f_L$. Using similar arguments, we obtain
\begin{eqnarray*}
\|\pi(f)-f_L\|^2 &=& \langle\pi(f)-f_L,\pi(f)-f_L\rangle=\langle f-f_L,f-f_L\rangle\\
&=& \|f-f_L\|^2,
\end{eqnarray*}
which completes the proof.
\end{proof}

With any function $f\colon\I^n\to\R$ we can associate the following symmetric function
$$
\mathrm{Sym}(f)=\frac{1}{n!}\sum_{\pi\in S_n}\pi(f).
$$
It follows immediately from Propositions~\ref{prop:asd7fd} and \ref{prop:sd9f7} that both functions $f$ and $\mathrm{Sym}(f)$ have the same best
shifted $L$-statistic approximation $f_L$. Combining this observation with Proposition~\ref{prop:sd9f7}, we derive immediately the following
corollary.

\begin{corollary}\label{cor:f7as8}
For every $k\in [n]$, every $f\in L^2(\I^n)$, and every $\pi\in S_n$, we have $I(f,k)=I(\pi(f),k)=I(\mathrm{Sym}(f),k)$.
\end{corollary}

\begin{remark}\label{rem:as98d7f}
Corollary~\ref{cor:f7as8} shows that, to compute $I(f,k)$, we can replace $f$ with $\mathrm{Sym}(f)$. For instance, if $f(\bfx)=x_i$ for some
$i\in [n]$ then $\mathrm{Sym}(f)=\frac{1}{n}\sum_{i=1}^nx_i=\frac{1}{n}\sum_{i=1}^nx_{(i)}$ and hence, using Proposition~\ref{prop:sdf79}, we
obtain $I(f,k)=\frac 1n$.
\end{remark}

We say that two functions $f\colon\I^n\to\R$ and $g\colon\I^n\to\R$ are \emph{symmetrically equivalent} (and we write $f\sim g$) if
$\mathrm{Sym}(f)=\mathrm{Sym}(g)$. By Corollary~\ref{cor:f7as8}, for any $f,g\in L^2(\I^n)$ such that $f\sim g$, we have $I(f,k)=I(g,k)$.

We end this section by analyzing the behavior of the influence index $I(f,k)$ on some special classes of functions.

Given $k\in [n]$, we say that the order statistic $x_{(k)}$ is \emph{ineffective} almost everywhere for a function $f\colon\I^n\to\R$ if
$\Delta_{(k),y-x_{(k)}}f(\bfx)=0$ for almost all $\bfx\in\cup_{\pi\in S_n}\I^n_{\pi}$ and almost all $y\in\left]x_{(k-1)},x_{(k+1)}\right[$. For instance,
given unary functions $f_1,f_2\in L^2(\I)$, the order statistic $x_{(1)}$ is ineffective almost everywhere for the function $f\colon\I^2\to\R$ such that
$$
f(x_1,x_2)=
\begin{cases}
f_1(x_1), & \mbox{if $x_1>x_2$,}\\
f_2(x_2), & \mbox{if $x_1<x_2$.}
\end{cases}
$$

The following result immediately follows from Theorem~\ref{thm:s7a897}.

\begin{proposition}
Let $k\in [n]$ and $f\in L^2(\I^n)$. If $x_{(k)}$ is ineffective almost everywhere for $f$, then $I(f,k)=0$.
\end{proposition}

The \emph{dual} of a function $f\colon\I^n\to\R$ is the function $f^d\colon\I^n\to\R$ defined by
$f^d(\bfx)=1-f(\mathbf{1}_{[n]}-\bfx)$. A function $f\colon\I^n\to\R$ is said to be \emph{self-dual} if $f^d=f$. By using the change of
variables theorem, we immediately derive the following result.

\begin{proposition}
For every $f\in L^2(\I^n)$ and every $k\in [n]$, we have $I(f^d,k)=I(f,n-k+1)$. In particular, if $f$ is self-dual, then $I(f,k)=I(f,n-k+1)$.
\end{proposition}

\section{Alternative expressions for the index}

The computation of the index $I(f,k)$ by means of (\ref{eq:dfdgf78}) or (\ref{eq:sdf79}) might be not very convenient due to the presence of the
order statistic functions. To make those integrals either more tractable or easier to evaluate numerically, we provide in this section some
alternative expressions for the index $I(f,k)$ that do not involve any order statistic.

We first derive useful formulas for the computation of the integral $\langle f,\mathrm{os}_k\rangle$ (Proposition~\ref{prop:df46}). To this
extent, we consider the following direct generalization of order statistic functions.

\begin{definition}
For every nonempty $S=\{i_1,\ldots,i_s\}\subseteq [n]$, $s=|S|$, and every $k\in [s]$, we define the function $\mathrm{os}_{k:S}\colon
\I^n\to\R$ as $\mathrm{os}_{k:S}(\bfx)=\mathrm{os}_{k:s}(x_{i_1},\ldots,x_{i_s})$.
\end{definition}

To simplify the notation, we will write $x_{k:S}$ for $\mathrm{os}_{k:S}(\bfx)$. Thus $x_{k:S}$ is the $k$th order statistic of the variables in
$S$.

\begin{lemma}\label{lemma:sd7f}
For every $s\in [n]$ and every $k\in [s]$, we have
\begin{equation}\label{eq:0sf98}
\sum_{\textstyle{S\subseteq [n]\atop |S|=s}} x_{k:S} = \sum_{j=k}^n{j-1\choose k-1}{n-j\choose s-k}\, x_{j:n}.
\end{equation}
\end{lemma}

\begin{proof}
Since both sides of (\ref{eq:0sf98}) are symmetric and continuous functions on $\I^n$, we can assume $x_1<\cdots <x_n$. Then, for every $j\in
[n]$, we have $x_{k:S}=x_j$ if and only if $S\ni j$ and $|S\cap[j-1]|=k-1$. The result then follows by counting those sets $S$ of cardinality
$s$ and having these two properties.
\end{proof}

\begin{lemma}\label{lemma:dfs7}
For every $k\in [n]$, we have
\begin{eqnarray}
x_{k:n} &=& \sum_{\textstyle{S\subseteq [n]\atop |S|\geqslant k}}(-1)^{|S|-k}{|S|-1\choose k-1}\, x_{|S|:S}\label{eq:0sf98a}\\
x_{k:n} &=& \sum_{\textstyle{S\subseteq [n]\atop |S|\geqslant n-k+1}}(-1)^{|S|-n+k-1}{|S|-1\choose n-k}\, x_{1:S}\label{eq:0sf98b}
\end{eqnarray}
\end{lemma}

\begin{proof}
By using (\ref{eq:0sf98}), we can rewrite the right-hand side of (\ref{eq:0sf98a}) as
\begin{eqnarray*}
\sum_{s=k}^n(-1)^{s-k}{s-1\choose k-1}\, \sum_{\textstyle{S\subseteq [n]\atop |S|=s}}x_{s:S} &=& \sum_{s=k}^n(-1)^{s-k}{s-1\choose
k-1}\,\sum_{j=s}^n{j-1\choose s-1}\, x_{j:n}\\
&=& \sum_{j=k}^n{j-1\choose k-1}x_{j:n}\,\sum_{s=k}^j(-1)^{s-k}{j-k\choose s-k}
\end{eqnarray*}
where the inner sum equals $(1-1)^{j-k}$. This proves (\ref{eq:0sf98a}). Identity (\ref{eq:0sf98b}) then follows by dualization.
\end{proof}

We now provide four formulas for the computation of the integral $\int_{\I^n}f(\bfx)\, x_{(k)}\, d\bfx$. From these formulas we will easily
derive alternative expressions for the index $I(f,k)$.

\begin{proposition}\label{prop:df46}
For every function $f\in L^2(\I^n)$ and every $k\in [n]$, the integral $J_{k:n}=\int_{\I^n}f(\bfx)\, x_{(k)}\, d\bfx$ is given by each of the
following expressions:
\begin{eqnarray}
& \int_{\I^n}f(\bfx)\, d\bfx-\sum_{S\subseteq [n]:|S|\geqslant k}(-1)^{|S|-k}{|S|-1\choose
k-1}\int_0^1\int_{[0,y]^S}\int_{[0,1]^{[n]\setminus S}}f(\bfx)\, d\bfx\, dy & \label{eq:dfsg1}\\
& \sum_{S\subseteq [n]:|S|\geqslant n-k+1}(-1)^{|S|-n+k-1}{|S|-1\choose
n-k}\int_0^1\int_{[y,1]^S}\int_{[0,1]^{[n]\setminus S}}f(\bfx)\, d\bfx\, dy & \label{eq:dfsg2}\\
& \int_{\I^n}f(\bfx)\, d\bfx-\sum_{S\subseteq [n]:|S|\geqslant k}\int_0^1\int_{[0,y]^S}\int_{[y,1]^{[n]\setminus S}}f(\bfx)\, d\bfx\, dy & \label{eq:dfsg3}\\
& \sum_{S\subseteq [n]:|S|< k}\int_0^1\int_{[0,y]^S}\int_{[y,1]^{[n]\setminus S}}f(\bfx)\, d\bfx\, dy & \label{eq:dfsg4}
\end{eqnarray}
\end{proposition}

\begin{proof}
By linearity of the integrals, we can assume that $f$ has nonnegative values. Then, we define the measure $\mu_f$ as $\mu_f(A)=\int_{A}f(\bfx)\, d\bfx$ for every Borel subset $A$ of $\I^n$. To compute integral
$J_{k:n}$, we can use Lemma~\ref{lemma:dfs7} and compute only the integrals $J_{|S|:S}=\int_{\I^n}f(\bfx)\vee_{i\in S}x_i\,d\bfx$ and
$J_{1:S}=\int_{\I^n}f(\bfx)\wedge_{i\in S}x_i\,d\bfx$. To this extent, we define
$$
F_{|S|:S}(y) ~=~ \mu_f(\{\bfx\in\I^n : \vee_{i\in S}x_i\leqslant y\})~=~\int_{[0,y]^S}\int_{[0,1]^{[n]\setminus S}}f(\bfx)\, d\bfx
$$
and
\begin{eqnarray*}
F_{1:S}(y) &=& \mu_f(\{\bfx\in\I^n : \wedge_{i\in S}x_i\leqslant y\})~=~\mu_f(\I^n)-\mu_f(\{\bfx\in\I^n : \wedge_{i\in S}x_i> y\})\\
&=& \int_{\I^n}f(\bfx)\, d\bfx - \int_{[y,1]^S}\int_{[0,1]^{[n]\setminus S}}f(\bfx)\, d\bfx
\end{eqnarray*}
We then have
$$
J_{|S|:S}=\int_{\I^n}\vee_{i\in S}x_i\,d\mu_f=\int_0^1 y\, dF_{|S|:S}(y) = \lim_{y\to 1}F_{|S|:S}(y)-\int_0^1F_{|S|:S}(y)\, dy
$$
and similarly for $J_{1:S}$. This proves (\ref{eq:dfsg1}) and (\ref{eq:dfsg2}).

We prove (\ref{eq:dfsg3}) similarly by considering
\begin{eqnarray*}
F_{k:n}(y) &=& \mu_f(\{\bfx\in\I^n : x_{(k)}\leqslant y\})\\
&=& \mu_f(\cup_{|S|\geqslant k}\{\bfx\in\I^n : \vee_{i\in S}x_i\leqslant y < \wedge_{i\in [n]\setminus S}x_i\})\\
&=& \sum_{|S|\geqslant k}\int_0^1\int_{[0,y]^S}\int_{[y,1]^{[n]\setminus S}}f(\bfx)\, d\bfx\, dy.
\end{eqnarray*}
Finally, we prove (\ref{eq:dfsg4}) by observing that
\begin{multline*}
\int_{\I^n}f(\bfx)\, d\bfx ~=~ \mu_f(\I^n) ~=~ \mu_f(\cup_{|S|\geqslant 0}\{\bfx\in\I^n : \vee_{i\in S}x_i\leqslant y < \wedge_{i\in [n]\setminus S}x_i\})\\
~=~ \sum_{|S|\geqslant 0}\int_0^1\int_{[0,y]^S}\int_{[y,1]^{[n]\setminus S}}f(\bfx)\, d\bfx\, dy.\qedhere
\end{multline*}
\end{proof}

\begin{remark}
In the special case when $f$ is a probability density function on $\I^n$, the integral $\int_{\I^n}f(\bfx)\, x_{(k)}\, d\bfx$ is precisely the expected value $E_f(X_{k:n})$, which is well investigated in statistics (see \cite{BalRao98,DavNag03}).
\end{remark}

From Definition~\ref{de:dfdgf78} and Proposition~\ref{prop:df46}, we derive the following three formulas. The computations are
straightforward and thus omitted.

\begin{eqnarray}
& \frac{I(f,k)}{(n+1)(n+2)}=\sum_{S\subseteq [n]:|S|\geqslant k-1}(-1)^{|S|+1-k}{|S|+1\choose
k}\int_0^1\int_{[0,y]^S}\int_{[0,1]^{[n]\setminus S}}f(\bfx)\, d\bfx\, dy & \label{eq:dfsg5}\\
& \frac{I(f,k)}{(n+1)(n+2)}=\sum_{S\subseteq [n]:|S|\geqslant n-k}(-1)^{|S|-n+k-1}{|S|+1\choose
n-k+1}\int_0^1\int_{[y,1]^S}\int_{[0,1]^{[n]\setminus S}}f(\bfx)\, d\bfx\, dy & \label{eq:dfsg6}\\
& \frac{I(f,k)}{(n+1)(n+2)}=\big(\sum_{S\subseteq [n]:|S|=k-1}-\sum_{S\subseteq [n]:|S|=k}\big)\int_0^1\int_{[0,y]^S}\int_{[y,1]^{[n]\setminus S}}f(\bfx)\, d\bfx\, dy.\label{eq:dfsg7}
\end{eqnarray}

\section{Some examples}

We now apply our results to two special classes of functions, namely the multiplicative functions and the Lov\'asz extensions of pseudo-Boolean
functions. The latter class includes the so-called discrete Choquet integrals, well-known in aggregation function theory.

\subsection{Multiplicative functions}

Consider the function $f(\bfx)=\prod_{i=1}^n\varphi_i(x_i)$, where $\varphi_i\in L^2(\I)$, and set $\Phi_i(x)=\int_0^x\varphi_i(t)\, dt$  for
$i=1,\ldots,n$. By using (\ref{eq:dfsg5}), we obtain
\begin{equation}\label{eq:xcvx09c}
\frac{I(f,k)}{(n+1)(n+2)}=\sum_{\textstyle{S\subseteq [n]\atop |S|\geqslant k-1}}(-1)^{|S|+1-k}{|S|+1\choose k}\prod_{i\in [n]\setminus
S}\Phi_i(1)\int_0^1\prod_{i\in S}\Phi_i(y)\, dy
\end{equation}

The following result gives a concise expression for $I(f,k)$ when $f$ is symmetric.

\begin{proposition}\label{prop:asf98}
Let $f\colon\I^n\to\R$ be given by $f(\bfx)=\prod_{i=1}^n\varphi(x_i)$, where $\varphi\in L^2(\I)$, and let $\Phi(x)=\int_0^x\varphi(t)\, dt$.
Then, for every $k\in [n]$, we have
$$
I(f,k)=
\begin{cases}
\Phi(1)^n\int_0^1 D_z h(z;k+1,n-k+2)|_{z=\Phi(y)/\Phi(1)}\, dy, & \mbox{if $\Phi(1)\neq 0$,}\\
(-1)^{n-k+1}(n+1)\,\frac{\Gamma(n+3)}{\Gamma(k+1)\,\Gamma(n-k+2)}\,\int_0^1\Phi(y)^n\, dy, & \mbox{if $\Phi(1)=0$,}
\end{cases}
$$
where $h(z;a,b)=z^{a-1}(1-z)^{b-1}/B(a,b)$ is the probability density function of the beta distribution with parameters $a$ and $b$.
\end{proposition}

\begin{proof}
Suppose that $\Phi(1)\neq 0$. By using (\ref{eq:dfsg7}), we obtain
\begin{multline*}
\frac{I(f,k)}{(n+1)(n+2)} = \Bigg(\sum_{\textstyle{S\subseteq [n]\atop |S|=k-1}}-\sum_{\textstyle{S\subseteq [n]\atop
|S|=k}}\Bigg)\int_0^1\Phi(y)^{|S|}\big(\Phi(1)-\Phi(y)\big)^{n-|S|}\, dy\\
= \Phi(1)^n\int_0^1\Big({n\choose k-1}z^{k-1}(1-z)^{n-k+1}-{n\choose k}z^k(1-z)^{n-k}\Big)\Big|_{z=\Phi(y)/\Phi(1)}\, dy,
\end{multline*}
which proves the result. The case $\Phi(1)= 0$ follows from (\ref{eq:xcvx09c}).
\end{proof}

\begin{example}\label{ex:0sd7}
Let $f\colon\I^n\to\R$ be given by $f(\bfx)=\big(\prod_{i=1}^n x_i\big)^c$, where $c>-\frac 12$. For instance, the product function corresponds to $c=1$ and the geometric mean function to $c=1/n$. We can calculate $I(f,k)$ by using
Proposition~\ref{prop:asf98} with $\varphi(x)=x^c$. Using the substitution $z=y^{c+1}$ and then integrating by parts, we obtain
$$
I(f,k) ~=~ c\,\Big(\frac{1}{c+1}\Big)^{n+2}\,\frac{\Gamma(n+3)\,\Gamma(k-1+\frac{1}{c+1})}{\Gamma(k+1)\,\Gamma(n+1+\frac{1}{c+1})}%
~=~ \frac{\Gamma(k-1+\frac{1}{c+1})}{\Gamma(k+1)\,\Gamma(\frac{1}{c+1})}\, I(f,1),
$$
with
$$
I(f,1)= c\,\big(\frac{1}{c+1}\big)^{n+2}\,\frac{\Gamma(n+3)\,\Gamma(\frac{1}{c+1})}{\Gamma(n+1+\frac{1}{c+1})}\,.
$$
We observe that $I(f,k)\to I(f,1)$ as $c\to -\frac 12$. Also, for $c>0$, we have $I(f,k+1)<I(f,k)$ for every $k\in [n-1]$. As expected in this case, the smallest variables are more
influent on $f$ than the largest ones.
\end{example}

\subsection{Lov\'asz extensions}

Recall that an $n$-place (\emph{lattice}) \emph{term function} $p\colon\I^n\to\I$ is a combination of projections $\mathbf{x}\mapsto x_i$ $(i\in
[n])$ using the fundamental lattice operations $\wedge$ and $\vee$; see \cite{BurSan81}. For instance,
$$
p(x_1,x_2,x_3)=(x_1\wedge x_2)\vee x_3
$$
is a $3$-place term function. Note that, since $\I$ is a bounded chain, here the lattice operations $\wedge$ and $\vee$ reduce to the minimum and
maximum functions, respectively.

Clearly, any shifted linear combination of $n$-place term functions
$$
f(\bfx)=c_0+\sum_{i=1}^mc_i\, p_i(\bfx)
$$
is a continuous function whose restriction to any standard simplex $\I^n_{\pi}$ $(\pi\in S_n)$ is a shifted linear function. According to
Singer~\cite[{\S}2]{Sin84}, $f$ is then the \emph{Lov\'asz extension} of the pseudo-Boolean function $f|_{\{0,1\}^n}$, that is, the
continuous function $f\colon\I^n\to\R$ which is defined on each standard simplex $\I^n_{\pi}$ as the unique affine function that
coincides with $f|_{\{0,1\}^n}$ at the $n+1$ vertices of $\I^n_{\pi}$. Singer showed that a Lov\'asz extension can always be written as
\begin{equation}\label{eq:8a0wedsaf}
f(\bfx)=f_{n+1}^{\pi}+\sum_{i=1}^n (f_i^{\pi}-f_{i+1}^{\pi})\, x_{\pi(i)}\qquad (\bfx\in\I^n_{\pi}),
\end{equation}
with $f_i^{\pi}=f(\mathbf{1}_{\{\pi(i),\ldots,\pi(n)\}})=v_f(\{\pi(i),\ldots,\pi(n)\})$ for $i\in [n+1]$, where the set function $v_f\colon
2^{[n]}\to\R$ is defined as $v_f(S)=f(\mathbf{1}_S)$. In particular, $f_{n+1}^{\pi}=c_0=f(\mathbf{0})$. Conversely, any continuous function
$f\colon\I^n\to\R$ that reduces to an affine function on each standard simplex is a shifted linear combination of term functions:
\begin{equation}\label{eq:8a0dsaf}
f(\bfx)=\sum_{S\subseteq [n]} m_f(S)\, x_{1:S},
\end{equation}
where $m_f\colon 2^{[n]}\to\R$ is the \emph{M\"obius transform} of $v_f$, defined as
$$
m_f(S)=\sum_{T\subseteq S}(-1)^{|S|-|T|}\, v_f(T).
$$
Indeed, expression (\ref{eq:8a0dsaf}) reduces to an affine function on each standard simplex and agrees with $f(\mathbf{1}_S)$ at $\mathbf{1}_S$
for every $S\subseteq [n]$. Thus the class of shifted linear combinations of $n$-place term functions is precisely the class of $n$-place
Lov\'asz extensions.

\begin{remark}
A nondecreasing Lov\'asz extension $f\colon\I^n\to\R$ such that $f(\mathbf{0})=0$ is also called a \emph{discrete Choquet integral}. For general
background, see for instance \cite{GraMarMesPap09}.
\end{remark}

For every nonempty $S\subseteq [n]$ and every $k\in [|S|]$, the function $\mathrm{os}_{k:S}$ is a Lov\'asz extension and, from
(\ref{eq:0sf98a}), we have
$$
x_{k:S} = \sum_{\textstyle{T\subseteq S\atop |T|\geqslant k}}(-1)^{|T|-k}{|T|-1\choose k-1}\, x_{|T|:T}
$$

The following proposition gives a concise expression for the index $I(\mathrm{os}_{j:S},k)$. We first consider a lemma.

\begin{lemma}\label{lemma:dsd78s}
For every nonempty $S\subseteq [n]$ and every $j\in [|S|]$, we have
$$
\mathrm{Sym}(\mathrm{os}_{j:S})=\frac{1}{{n\choose |S|}}\sum_{\textstyle{T\subseteq [n]\atop |T|=|S|}}\mathrm{os}_{j:T}.
$$
\end{lemma}

\begin{proof}
It is easy to see that $\mathrm{Sym}(\mathrm{os}_{j:S})(\bfx)=\frac{1}{n!}\sum_{\pi\in S_n}\mathrm{os}_{j:\pi(S)}(\bfx)$. This proves the result
for there are exactly $|S|!(n-|S|)!$ permutations that map $S$ to a given set $T\subseteq [n]$ such that $|T|=|S|$.
\end{proof}

\begin{proposition}
For every nonempty $S\subseteq [n]$, every $j\in [|S|]$, and every $k\in [n]$, we have
\begin{equation}\label{eq:asdf79}
I(\mathrm{os}_{j:S},k)=\frac{{k-1\choose j-1}{n-k\choose |S|-j}}{{n\choose |S|}}
\end{equation}
if $0\leqslant k-j\leqslant n-|S|$, and $0$, otherwise.
\end{proposition}

\begin{proof}
The result follows from Corollary~\ref{cor:f7as8}, Lemma~\ref{lemma:sd7f}, and Lemma~\ref{lemma:dsd78s}.\footnote{The right-hand side of
(\ref{eq:asdf79}) can also be viewed as the multivariate hypergeometric distribution $ \textstyle{{1\choose 1}{k-1\choose j-1}{n-k\choose
|S|-j}/{n\choose |S|}} $ obtained directly from the proof of Lemma~\ref{lemma:sd7f}.}
\end{proof}

The following proposition gives an explicit expression for the index $I(f,k)$ when $f$ is a Lov\'asz extension.

\begin{proposition}\label{prop:s80}
If $f\colon\I^n\to\R$ is a Lov\'asz extension, then
\begin{equation}\label{eq:8s798f}
f(\bfx)=f(\mathbf{0})+\sum_{i=1}^n x_{(i)}\, D_{(i)}f(\bfx).
\end{equation}
Moreover, for every $k\in [n]$, we have
\begin{equation}\label{eq:sb79s}
I(f,k)=\overline{v}_f(n-k+1)-\overline{v}_f(n-k)=\sum_{s=1}^{n-k+1}{n-k\choose s-1}\,\overline{m}_f(s)\, ,
\end{equation}
where $\overline{v}_f(s)={n\choose s}^{-1}\sum_{S\subseteq [n]: |S|=s}v_f(S)$ and $\overline{m}_f(s)={n\choose s}^{-1}\sum_{S\subseteq [n]:
|S|=s}m_f(S)$.
\end{proposition}

\begin{proof}
Let $f\colon\I^n\to\R$ be a Lov\'asz extension and let $k\in [n]$. For every $\pi\in S_n$ and every $\bfx\in\I^n_{\pi}$, from
(\ref{eq:8a0wedsaf}) if follows that $D_{\pi(k)}f(\bfx)=f_k^{\pi}-f_{k+1}^{\pi}$. This estalishes (\ref{eq:8s798f}).

By Proposition~\ref{prop:sdf79}, we then obtain
\begin{multline*}
I(f,k) = \sum_{\pi\in S_n}\int_{\I^n_{\pi}}h_k(\bfx)\, D_{\pi(k)}f(\bfx)\, d\bfx\\
= (n+1)(n+2)\sum_{\pi\in S_n}(f_k^{\pi}-f_{k+1}^{\pi})\int_{\I^n_{\pi}}(x_{\pi(k+1)}-x_{\pi(k)})(x_{\pi(k)}-x_{\pi(k-1)})\, d\bfx
\end{multline*}
Since the integral is equal to $1/(n+2)!$, we obtain $ I(f,k) = \frac{1}{n!}\sum_{\pi\in S_n}(f_k^{\pi}-f_{k+1}^{\pi}), $ which, after some
algebra, leads to the first equality in (\ref{eq:sb79s}). Finally, by combining Proposition~\ref{prop:asd7fd} with (\ref{eq:8a0dsaf}) and
(\ref{eq:asdf79}) (for $j=1$), we completely establish (\ref{eq:sb79s}).
\end{proof}

\begin{remark}
A function $f\colon\I^n\to\R$ solves equation~(\ref{eq:8s798f}) if and only if, for every $\pi\in S_n$, the function
$f|_{\I^n_{\pi}}-f(\mathbf{0})$ is an eigenfunction of the Euler operator with eigenvalue $1$. Thus this function reduces to a homogeneous
function of degree $1$ whenever it is differentiable. Notice however that such a function need not be linear even if $f$ is continuous on
$\I^n$. For instance, the geometric mean $f(\bfx)=\prod_{i=1}^nx_i^{1/n}$ is a continuous function solving (\ref{eq:8s798f}).
\end{remark}

We can readily see that the
shifted $L$-statistic functions are precisely the symmetric Lov\'asz extensions. From this observation we derive the following result.

\begin{proposition}
For any Lov\'asz extension $f\colon\I^n\to\R$, we have $f_L=\mathrm{Sym}(f)$ and
$$
\mathrm{Sym}(f)=f(\mathbf{0})+\sum_{i=1}^n I(f,i)\,\mathrm{os}_i.
$$
\end{proposition}

\begin{proof}
Let $f\colon\I^n\to\R$ be a Lov\'asz extension. Then $\mathrm{Sym}(f)$ is a symmetric Lov\'asz extension or, equivalently, a shifted $L$-statistic
function. By Propositions~\ref{prop:asd7fd} and \ref{prop:sd9f7}, we have $f_L=\mathrm{Sym}(f)_L=\mathrm{Sym}(f)$. The result then follows since $f_L(\mathbf{0})=\mathrm{Sym}(f)(\mathbf{0})=f(\mathbf{0})$.
\end{proof}

\section{Applications}

We briefly discuss some applications of the influence index in aggregation theory and statistics. We also introduce a normalized version of the
index as well as the coefficient of determination of the approximation problem.

\subsection{Influence index in aggregation theory}

Several indexes (such as interaction, tolerance, and dispersion indexes) have been proposed and investigated in aggregation theory to better
understand the general behavior of aggregation functions with respect to their variables; see \cite[Chap.~10]{GraMarMesPap09}. These indexes
enable one to classify the aggregation functions according to their behavioral properties. The index $I(f,k)$ can also be very informative and
thus contribute to such a classification. As an example, we have computed this index for the arithmetic mean and geometric mean functions (see
Remark~\ref{rem:as98d7f} and Example~\ref{ex:0sd7}) and we can observe for instance that the smallest variable $x_{(1)}$ has a larger influence
on the latter function.

\begin{remark}
Noteworthy aggregation functions are the so-called conjunctive aggregation functions, that is, nondecreasing functions $f\colon\I^n\to\R$
satisfying $0\leqslant f(\bfx)\leqslant x_{(1)}$; see \cite[Chap.~3]{GraMarMesPap09}. Although these functions are bounded from above by
$x_{(1)}$, the index $I(f,k)$ need not be maximum for $k=1$. For instance, for the binary conjunctive aggregation function
$$
f(x_1,x_2)=
\begin{cases}
0, & \mbox{if $x_1\vee x_2<\frac 34$},\\
x_1\wedge x_2\wedge \frac 14, & \mbox{otherwise},
\end{cases}
$$
we have $I(f,1)=\frac{17}{128}$ and $I(f,2)=\frac{19}{64}$, and hence $I(f,1)<I(f,2)$.
\end{remark}

In the framework of aggregation functions, it can be natural to consider and identify the functions $f\in L^2(\I^n)$ for which the order
statistics are equally influent, that is, such that $I(f,k)=I(f,1)$ for all $k\in [n]$. As far as the Lov\'asz extensions are concerned, we have
the following result, which can be easily derived from Proposition~\ref{prop:s80} and the immediate identities
$$
\overline{v}_f(s)=\sum_{t=0}^s{s\choose t}\overline{m}_f(t)\quad\mbox{and}\quad\overline{m}_f(s)=\sum_{t=0}^s(-1)^{s-t}{s\choose t}\overline{v}_f(t).
$$

\begin{proposition}
If $f\colon\I^n\to\R$ is a Lov\'asz extension, then the following are equivalent.
\begin{enumerate}
\item[(a)] We have $I(f,k)=I(f,1)$ for all $k\in [n]$.

\item[(b)] The sequence $(\overline{v}_f(s))_{s=0}^n$ is in arithmetic progression.

\item[(c)] We have $\overline{m}_f(s)=0$ for $s=2,\ldots,n$.
\end{enumerate}
\end{proposition}

\subsection{Influence index in statistics}

It can be informative to assess the influence of every order statistic on a given statistic to measure, e.g., its behavior with respect to the
extreme values. From this information we can also approximate the given statistic by a shifted $L$-statistic. Of course, for $L$-statistics
(such as Winsorized means, trimmed means, linearly weighted means, quasi-ranges, Gini's mean difference; see \cite[{\S}6.3, {\S}8.8,
{\S}9.4]{DavNag03}), the computation of the influence indexes is immediate. However, for some other statistics such as the central moments, the
indexes can be computed via (\ref{eq:dfsg5})--(\ref{eq:dfsg7}).

\begin{example}
The closest shifted $L$-statistic to the variance $\sigma^2=\frac 1n\sum_{i=1}^n(X_i-\overline{X})^2$ is given by
$$
\sigma^2_L=\frac{1-n^2}{12n(n+3)}+\sum_{k=1}^nI(\sigma^2,k)\, X_{(k)},
$$
with $I(\sigma^2,k)=(n+2)(2k-n-1)/(n^2(n+3))$, which can be computed from (\ref{eq:dfsg7}).\footnote{In terms of Gini's mean difference
\cite[{\S}9.4]{DavNag03}, $G=\frac{2}{n(n-1)}\sum_{k=1}^n(2k-n-1)X_{(k)}$, we simply obtain $\sigma^2_L=\frac{n-1}{12n(n+3)}\,\big(6(n+2)\,
G-(n+1)\big)$.} We then immediately see that the smallest and largest variables are the most influent.
\end{example}

\subsection{Normalized index and coefficient of determination}

Coming back to the interpretation of the influence index as a covariance (see {\S}3), it is natural to consider the Pearson correlation
coefficient instead of that covariance. In this respect, we note that $\sigma^2(g_k)=E(g_k^2)=I(g_k,k)=2(n+1)(n+2)$, where the latter inequality
is immediate since $g_k\in V_L$.

\begin{definition}
The \emph{normalized influence index} is the mapping
$$
r\colon\{f\in L^2(\I^n): \mbox{$f$ is non constant}\}\times [n]\to\R
$$
defined by
$$
r(f,k)=\frac{I(f,k)}{\sigma(f)\,\sqrt{2(n+1)(n+2)}}
$$
\end{definition}

From this definition it follows that $-1\leqslant r(f,k)\leqslant 1$, where the bounds are tight. Moreover, this index remains unchanged under
interval scale transformations, that is, $r(af+b,k)=r(f,k)$ for all $a>0$ and $b\in\R$. Finally, we also have $r(f^d,k)=r(f,n-k+1)$.

The \emph{coefficient of determination} of the best shifted $L$-statistic approximation of a non constant function $f\in L^2(\I^n)$ is defined
by $R^2(f)=\sigma^2(f_L)/\sigma^2(f)$. We then have
$$
R^2(f)=\frac{1}{\sigma^2(f)}\,\sigma^2\bigg(\sum_{j=1}^{n+1}a_j\,
x_{(j)}\bigg)=\frac{1}{\sigma^2(f)}\,\mathbf{a}^T(M-\mathbf{c}\mathbf{c}^T)\,\mathbf{a},
$$
where $\mathbf{c}$ is the $(n+1)$st column of $M$.

\section*{Acknowledgments}

The authors wish to thank Samuel Nicolay for fruitful discussions. This research is supported by the internal research project F1R-MTH-PUL-09MRDO of the University of Luxembourg.


\end{document}